\documentclass[11pt]{amsart}
\usepackage{amsmath,amssymb,mathrsfs,color}

\let\cal=\mathcal

\def\R{{\mathbb R}}

\newtheorem{thm}{Theorem}[section]

\newtheorem{lem}[thm]{Lemma}
\newtheorem{prop}[thm]{Proposition}

\newtheorem{exa}[thm]{Example}
\theoremstyle{definition}
\newtheorem{de}[thm]{Definition}
\theoremstyle{remark}
\newtheorem{rem}[thm]{Remark}

\numberwithin{equation}{section}

\newcommand{\rmd}{{\rm d}}

\allowdisplaybreaks

\begin{document}

\title[Almost Periodic Solutions and Stable solutions for SDEs]{Almost Periodic Solutions and Stable Solutions for Stochastic Differential Equations}

\author{Yong Li}
\address{Y. Li: School of Mathematics and Statistics, Center for Mathematics and Interdisciplinary Sciences,
Northeast Normal University, Changchun 130024, P. R. China; School of Mathematics,
Jilin University, Changchun 130012, P. R. China}
\email{liyong@jlu.edu.cn}

\author{Zhenxin Liu}
\address{Z. Liu: School of Mathematical Sciences,
Dalian University of Technology, Dalian 116024, P. R. China}
\email{zxliu@dlut.edu.cn}

\author{Wenhe Wang}
\address{W. Wang: School of Basic Sciences,
Changchun University of Technology, Changchun 130012, P. R. China}
\email{die\_jahre@163.com}

\thanks{The first author is supported by National Research Program of China Grant 2013CB834100 and NSFC Grant 11171132;
the second author is supported by NSFC Grants 11271151, 11522104, and the startup and Xinghai Youqing funds from
Dalian University of Technology.}

\date{September 14, 2016}




\begin{abstract}
In this paper, we discuss the relationships between stability and almost periodicity for solutions of stochastic differential equations.
Our essential idea is to get stability of solutions or systems by some inherited properties of Lyapunov functions. Under suitable
conditions besides Lyapunov functions, we obtain the existence of almost periodic solutions in distribution.
\end{abstract}
\maketitle
\section{Introduction}
\setcounter{equation}{0}

In 1924--1926, Bohr founded the theory of almost periodic functions  \cite{B1,B2,B3}. Roughly speaking, an almost periodic function
means that it is periodic up to any desired level of accuracy. Since many differential equations
arising from physics and other fields admit almost periodic solutions, almost periodicity becomes an important property of dynamical systems
and is extensively studied in the area of differential equations and dynamical systems. We refer the reader to the books,
e.g. Amerio and Prouse \cite{AP}, Fink \cite{Fink74}, Levitan and Zhikov \cite{LZ},
Yoshizawa \cite{Y75} etc, for an exposition.

For deterministic differential equations, the existence of almost periodic solutions was studied under various stability assumptions.
Markov \cite{M} defined a kind of stability which implies almost periodicity.
Deysach and Sell \cite{DS} assumed that there exists one bounded uniformly stable solution. Miller \cite{Miller}
assumed the existence of one bounded totally stable solution. Seifert \cite{S66} proposed the so-called $\Sigma$-stability, while Sell
\cite{Sell1,Sell2} proposed the stability under disturbance from the hull; actually, these two concepts of stability
are equivalent. Coppel \cite{C} sharpened Miller's result without the uniqueness of solutions by using the properties of
asymptotically almost periodic functions; Yoshizawa \cite{Y69} developed the idea of Coppel and improved all the results
mentioned above. On the other hand, the Lyapunov's second method was employed to investigate the
existence of almost periodic solutions: Hale \cite{Hale64} and Yoshizawa \cite{Y64} assumed the existence of Lyapunov
functions for pairs of solutions to conclude the uniform asymptotic stability in the large of the bounded solution.

However, the various stability assumptions mentioned above are not easily verified directly in practice. It is known that some stabilities,
such as uniform stability and uniform asymptotic stability, can be characterized by Lyapunov functions. So it seems that it is a good idea to
give some explicit conditions on Lyapunov functions to study the existence of almost periodic solutions, as Hale and
Yoshizawa did in \cite{Hale64,Y64}. This is exactly what we are to do in the present paper for stochastic differential equations (SDE).

For the stochastically perturbed semilinear equations, almost periodic solutions were studied by assuming that
the linear part of these equations satisfies the property of exponential dichotomy; see Halanay \cite{Hal}, Morozan and Tudor \cite{MT},
Da Prato and Tudor \cite{DT}, and Arnold and Tudor \cite{AT}, among others. For general SDEs, V\^arsana \cite{V} studied asymptotical almost
periodic (weaker than almost periodic) solutions by assuming that the stochastic system is total stable.
Very recently, Liu and Wang \cite{LW} investigated
the almost periodic solutions to SDEs by the separation method.

This paper is organized as follows. Section 2 is a preliminary
section. Section 3 contains main results of this paper, in which we study almost periodic solutions for SDEs by mainly the Lyapunov function method.
In Section 4, we illustrate our results by some applications.


\section{Preliminaries}
Assume that $(M,d)$ is a complete metric space. Here is the definition
of $M$-valued almost periodic and uniform almost periodic functions
in the sense of Bohr:
\begin{de}\label{nou001}
(i) Assume $\varphi(\cdot):\mathbb{R}\rightarrow M$ is continuous. We say set $A\subset\mathbb{R}$ is {\em relatively dense}
in $\mathbb{R}$ if for any given $\epsilon>0$,
there exists $l=l(\epsilon)>0$, such that for every $a\in\mathbb{R}$,
$(a,a+l)\cap A\neq\emptyset$. If there is a set $T(\epsilon,\varphi)$ relatively dense
such that for any $\tau\in T(\epsilon,\varphi)$ we have
\begin{equation}\label{evenuse02}
\sup_{t\in\mathbb{R}}\rho(\varphi(t+\tau),\varphi(t))<\epsilon,
\end{equation}
then we say that the function $\varphi$ is {\em almost periodic}.

(ii) A continuous function $f(\cdot,\cdot):\mathbb{R}\times\mathbb{R}^d\rightarrow\mathbb{R}^d$
is {\em almost periodic in $t$ uniformly on compact sets} if for every compact set $S\subset\mathbb{R}^d$ there exists a
relatively dense set $T(\epsilon,f,S)$ such that for every $\tau\in T(\epsilon,f,S)$ we have:
\begin{equation}\label{evenuse02}
\sup_{(t,x)\in\mathbb{R}\times S}|f(t+\tau,x)-f(t,x)|<\epsilon.
\end{equation}
We also say such $f(t,x)$ is {\em uniformly almost periodic} for short.
\end{de}

Bochner \cite{Boch27, Boch62} gave an equivalent condition to Bohr's almost periodicity.
The above definition of uniform almost periodicity can be found in Yoshizawa's book \cite{Y75};
Seifert and Fink made another
definition of uniform almost periodicity (see Definition $2.1$ in \cite{Fink74}).



For sequence $\alpha=\{\alpha_n\}$, we denote $\lim_{n\rightarrow+\infty}\varphi(\cdot+\alpha_n)$
as $T_{\alpha}\varphi(\cdot)$ if it exists, and the mode of convergence will be specified at each use; the similar notation will be used for
$T_\alpha f(\cdot,\cdot)=\lim_{n\rightarrow+\infty}f(\cdot+\alpha_n,\cdot)$. For simplicity, we also denote $\varphi(\cdot+a)$ by $\varphi_a(\cdot)$
and $f(\cdot+ a,\cdot)$ by $f_a(\cdot,\cdot)$ for given $a\in\R$.

For $\mathbb{R}^d$-valued uniformly almost periodic function $f(t,x)$, we denote
\begin{equation*}
\begin{split}
H(f):= & \{g(t,x);\text{there is sequence $\alpha$ such that } T_\alpha f=g\\
& \quad \text{ uniformly on
$\mathbb{R}\times S$ for each compact set $S\subset\mathbb{R}^d$}\},
\end{split}
\end{equation*}
as the {\em hull of $f$}. The hull has the following properties:

\begin{prop}\label{uap-p2}
Let $f(t,x)$ be uniformly almost periodic. Then:
\begin{itemize}
\item[{\rm (i)}] any $g\in H(f)$ is also uniformly almost periodic and $H(g)=H(f)$;

\item[{\rm (ii)}] for any $g\in H(f)$, there exists a sequence $\alpha$ with $\alpha_n\to +\infty$ (or $\alpha_n\to -\infty$) such that
$T_{\alpha} f=g$ uniformly on $\R\times S$ for any compact $S\subset \R^d$;

\item[{\rm (iii)}] for any sequence $\alpha'$, there exists a subsequence $\alpha\subset \alpha'$ such that
$T_\alpha f$ exists uniformly on $\R\times S$ for any compact $S\subset\R^d$.

\end{itemize}
\end{prop}

We respectively denote $[0,+\infty)$ and $(-\infty,0]$ as $\mathbb{R}_+$ and $\mathbb{R}_-$,
and recall the definition of asymptotically almost periodic function valued in $M$ as follows.
\begin{de}\label{aap}
Suppose that function $f(\cdot):\mathbb{R}_+\rightarrow M$ is continuous
and there exists an almost periodic function
$\eta(\cdot):\mathbb{R}\rightarrow M$, such that
\begin{equation}\label{appart}
\lim_{t\rightarrow+\infty}d(f(t),\eta(t))=0.
\end{equation}
Then we say $f(t)$ is {\em aymptotically almost periodic} (a.a.p. in short) on $\mathbb{R}_+$.
The $\eta(t)$ in \eqref{appart} is called the {\em almost periodic part of $f$}. The function $f$
being a.a.p. on $\mathbb{R}_-$ can be defined similarly.
\end{de}

\begin{rem}[See \cite{LZ}, Chapter $1$]
If $f$ is a.a.p. on $\mathbb{R}_+$ or $\mathbb{R}_-$, then its almost periodic part is unique.
\end{rem}



\begin{lem}\label{uselessmust01}
The following statements are equivalent to $f$ being asymptotic almost periodic on $\R_+$:
\begin{itemize}
\item[(i)] For any sequence $\alpha'=\{\alpha'_n\}$ such that $\alpha'_n\rightarrow+\infty$, there exists suitable subsequence $\alpha\subset\alpha'$ such that $T_{\alpha}f(t)$ uniformly exists on $\mathbb{R}_+$.
\\
\item[(ii)] For any sequence $\alpha'=\{\alpha'_n\}$ such that $\alpha'_n\rightarrow+\infty$, there exists a subsequence
$\alpha\subset\alpha'$ and a constant $\sigma=\sigma(\alpha)>0$ such that $T_{\alpha}f$ exists pointwise on $\mathbb{R}_+$ and if sequences
$\delta >0$, $\beta\subset\alpha,\gamma\subset\alpha$ are such that
\[
T_{\delta+\beta}f=h_1 \quad\hbox{and}\quad T_{\delta+\gamma}f=h_2
\]
exist pointwisely on $\mathbb{R}_+$, then either $h_1\equiv h_2$ or
$\inf_{t\in\mathbb{R}_+}d(h_1(t),h_2(t))\geq 2\sigma$.
\end{itemize}
The similar results hold when $f$ is asymptotic almost periodic on $\R_-$.
\end{lem}

In this paper, we study the SDE:
\begin{equation}\label{tagA001}
\rmd X(t)=f(t,X(t))\rmd t+g(t,X(t))\rmd W(t),
\end{equation}
where $f(t,x)$ is an $\R^d$-valued continuous function,
$g(t,x)$ is a $(d\times m)$-matrix-valued continuous function, and $W$ is a standard
$m$-dimensional Brownian motion. And we usually assume the coefficients are uniformly almost periodic.
Note that almost periodicity is defined on the whole $\mathbb{R}$,
but the Brownian motions in SDEs usually defined on $\mathbb{R}_+$. So we need to introduce
two-sided Brownian motion:
for two independent Brownian motions $W_1(t)$, $W_2(t)$ on the probability space
$(\Omega, \mathcal{F}, {P})$, let
\[
W (t) =\left\{ \begin{array}{ll}
                 W_1(t), &  \hbox{ for } t\ge 0, \\
                 -W_2(-t), &  \hbox{ for } t\le 0.
               \end{array}
 \right.
\]
Then $W$ is a two-sided Brownian motion defined on the filtered probability space
$(\Omega, \mathcal{F}, \mathbf{P},\mathcal{F}_t )$
with $\mathcal F_t=\sigma\{W(u): u \le t\}, t\in\R$.

Furthermore, we always assume \eqref{tagA001}'s coefficients satisfy the
following condition:

\medskip
\noindent {\bf (H)} The functions $f$, $g$ are uniformly almost periodic. And there
exists a constant $K>0$ such that, for every $t\in\mathbb{R}$ and $x,y\in\mathbb{R}^d$,
$$|f(t,x)-f(t,y)|\vee|g(t,x)-g(t,y)|\leq K|x-y|,$$
where $a\vee b=\max\{a,b\}$ for $a,b\in\mathbb{R}$.\\

For SDE \eqref{tagA001} satisfying condition {\bf (H)}, if there exists a sequence $\alpha$
such that $T_\alpha f=\tilde f$ and $T_\alpha g=\tilde g$, we denote the SDE with coefficients
$(T_\alpha f, T_\alpha g)$ as $(\tilde f,\tilde g)\in H(f,g)$ or
$T_\alpha(f,g)=(\tilde f,\tilde g)$ for short. Besides, by the definition of uniform almost periodic function,
if coefficients of \eqref{tagA001} satisfy the condition {\bf (H)}, they must satisfy the global linear
growth condition, that is, there is some constant $\hat{K}>0$, such that
$$|f(t,x)|\vee|g(t,x)|\leq \hat{K}(1+|x|^2)\text{, $\forall t\in\mathbb{R}$, $\forall x\in\mathbb{R}^d$.}$$

For $\mathbb{R}^d$-valued random variable $X$ on the probability
space $(\Omega,\mathcal{F},\mathbf{P})$, we denote $\mathcal{L}(X)$ as the distribution (or law) of $X$ on $\mathbb{R}^d$.
We denote by $\mathcal P(\R^d)$ the space of all Borel probability measures on $\R^d$.
For an $\R^d$-valued random variable $X$ or stochastic process $Y(t)$, we define the following norms:
$$\|X\|_2:=(\int_{\Omega}|X(\omega)|^2\rmd\mathbf{P}(\omega))^{\frac{1}{2}},\quad \|Y(t)\|_{\infty}:=\sup_{t}\|Y(t)\|_2.$$
In what follows, we denote:
\begin{align*}
& L^2(P,\R^d):=\{X: \|X\|_2 <\infty\}, \quad \mathcal{B}_r:=\{X\in L^2(P,\R^d): \|X\|_2 \le r\}, \\
& \mathcal{D}_r:=\{ \mu\in\mathcal{P}(\R^d): \exists X\in\mathcal{B}_r\text{ such that $\mathcal{L}(X)=\mu$} \}, \\
&\mathcal{B}_r^{\eqref{tagA001}}=\mathcal{B}_r^{(f,g)} :=\{X(\cdot): (X,W) \hbox{ weakly solves equation } (f,g)  \hbox{ on } \R \\
&\qquad\qquad\qquad\qquad \hbox{ on some filtered probability space for some } W \hbox{ and } \|X\|_\infty\le r \},\\
&\mathcal{D}_r^{\eqref{tagA001}}=\mathcal{D}_r^{(f,g)} :=\{\mu: \mu(\cdot) =\mathcal{L} (X(\cdot)) \hbox{ for some } X \in \mathcal{B}_r^{(f,g)} \},\\
&\mathcal{B}^{\eqref{tagA001}}=\mathcal{B}^{(f,g)}=\cup_{r>0}\mathcal{B}_r^{(f,g)}\text{, }
\mathcal{D}^{\eqref{tagA001}}=\mathcal{D}^{(f,g)}=\cup_{r>0}\mathcal{D}_r^{(f,g)}.
\end{align*}

We focus on the almost periodicity of distributions of SDEs' solutions instead of solutions
themselves. It's well known that $\mathcal{P}(\mathbb{R}^d)$
can be metrized with some distance
(which we denote as $\rho(\cdot,\cdot)$), such that the convergence under
distance $\rho(\cdot,\cdot)$ is equivalent to the convergence under the
weak-* topology of $\mathcal{P}(\mathbb{R}^d)$, and $\mathcal{P}(\mathbb{R}^d)$ is a
complete metric space under $\rho(\cdot,\cdot)$ (see \cite[Theorem $2.6.2$]{P} for details).

For a $\mathcal{P}(\mathbb{R}^d)$-valued continuous function $f$, one of the  necessary conditions
of the almost periodicity of $f$ is that,
the set $\{f(t);t\in\mathbb{R}\}$ is contained in some compact set. Naturally we need to consider
distributions of solutions for SDEs within some compact set. We get
compactness on the space $\mathcal{P}(\mathbb{R}^d)$ by $\mathbf{L}^2$-boundedness (see \cite{Pro} for details).

We define the uniform stability of distributions of solutions for SDEs as follows:

\begin{de}\label{h}
$\forall t_0\in\mathbb{R}$, we say element $\mu(t)\in\mathcal{D}^{\eqref{tagA001}}_r$ is
{\em uniformly stable on $[t_0,+\infty)$ within $\mathcal{D}^{\eqref{tagA001}}_r$} if for every $\epsilon>0$,
there exists $\delta=\delta(\epsilon)>0$ such that for any $t_1\geq t_0$ and any other element
$\eta(t)\in\mathcal{D}^{\eqref{tagA001}}_r$ satisfying
$$\rho(\mu(t_1),\eta(t_1))<\delta,$$
we have
\begin{equation*}
\sup_{t\in[t_1,+\infty)}\rho(\mu(t),\eta(t))<\epsilon.
\end{equation*}
If $\mu(t)$ is uniformly stable on $[t_1,+\infty)$ for every $t_1\in\mathbb{R}$, we call it
{\em uniformly stable} for short.
\end{de}

In what follows, we get the stability of solutions' distributions mainly by Lyapunov functions, which satisfy the following condition:
\medskip

\noindent {\bf (L)}
Assume that $V(\cdot,\cdot):\mathbb{R}\times\mathbb{R}^d\rightarrow\mathbb{R}_+$ is a function
$C^2$ in $t\in\mathbb{R}$, $C^3$ in $x\in\mathbb{R}^d$.
Assume that the differentials $D^iV$ of $V$ for $i=0,1,2$ and the derivatives $V_{tx_ix_j}$, $V_{x_ix_jx_k}$ for $i,j,k=1,2,\cdots,d$
are bounded on $\mathbb{R}\times S$ for every compact set $S\subset\mathbb{R}^d$. Furthermore,
\begin{equation}\label{whether01}
\inf_{t\in \R}V(t,x)>0\text{ for each $x\neq 0$, and $V(t,0)=0$ for all $t\in \R$.}
\end{equation}

\section{Main Results}

In this paper, we need following results from \cite{LW} for further discussion:

\begin{prop}[\cite{LW}, Theorem 3.1]\label{key}
Consider the following family of It\^o stochastic equations on $\R^d$
\begin{equation}
\rmd X = f_n(t,X)\rmd t + g_n(t,X)\rmd W,\quad n=1,2,\cdots,
\end{equation}
where $f_n$ are $\R^d$-valued, $g_n$ are $(d\times m)$-matrix-valued, and
$W$ is a standard $m$-dimensional Brownian motion.
Assume that $f_n, g_n$ satisfy condition {\bf (H)}. Assume further that $f_n\to f$, $g_n\to g$
point-wise on $\mathbb R\times \mathbb R^d$ as $n\to\infty$, and that
$X_n(t)\in \mathcal{B}^{(f_n,g_n)}_{r}$ for some constant $r>0$, independent of $n$. Then there is a
subsequence of $\{X_n\}$ which
converges in distribution, uniformly on compact intervals, to some
$X(t)\in \mathcal{B}^{(f, g)}_{r}$.
\end{prop}

\begin{prop}[\cite{LW}, Lemma $4.1$]\label{im}
Consider SDE \eqref{tagA001} with coefficients satisfying condition {\bf (H)}.
If SDE \eqref{tagA001} admits an $L^2$-bounded solution $X(t)$ on $\R$ which is
asymptotically almost periodic in distribution on $\R_+$, then SDE \eqref{tagA001} admits a solution
$Y$ on $\R$ which is
almost periodic in distribution such that
\[
\lim_{t\to+\infty} \rho (\mathcal{L}(X(t)), \mathcal{L}(Y(t))) =0 \quad \hbox{and} \quad \sup_{t\in\mathbb{R}}\mathbf{E}|Y(t)|^2 \le \sup_{t\in\mathbb{R}}\mathbf{E}|X(t)|^2.
\]
In particular, $\mathcal{L}(Y)$ is the almost periodic part of $\mathcal{L}(X)$.  The similar
result holds when $X$ is asymptotically
almost periodic in distribution on $\R_-$.
\end{prop}


Consider \eqref{tagA001} and let $V$ satisfy condition {\bf (L)}. For
$t\in\mathbb{R}$ and $x,y\in\mathbb{R}^d$, denote
\begin{equation}\label{mathscr01}
\begin{split}
\mathscr{L}V(t,x-y):= & \frac{\partial V}{\partial t}(t,x-y)+\sum_{i=1}^d\frac{\partial V}{\partial x_i}(t,x-y)(f_i(t,x)-f_i(t,y))\\
& +\frac{1}{2}\sum_{l=1}^m\sum_{i,j=1}^d(g_{il}(t,x)-g_{il}(t,y))\frac{\partial^2 V}{\partial x_i\partial x_j}(t,x-y)\\
& \cdot (g_{jl}(t,x)-g_{jl}(t,y))).
\end{split}
\end{equation}

Now we give a sufficient condition to the uniform stability in distribution we defined in Definition
\ref{h}:

\begin{thm}\label{oldfind01}
Suppose that \eqref{tagA001}'s coefficients satisfy condition {\bf (H)} and there is a function
$V(\cdot,\cdot)$ satisfying condition {\bf (L)}. Assume that there exists
some constant $b>0$ such that for all $(t,x)\in\mathbb{R}\times\mathbb{R}^d$,
\begin{equation}\label{oldfind02}
V(t,x)\leq b|x|^2,
\end{equation}
\begin{equation}\label{oldfind03}
\mathscr{L}V(t,x-y)\leq 0.
\end{equation}
Then if $\mathcal{D}^{\eqref{tagA001}}_r\neq\emptyset$ for some $r>0$, all the elements of it
are uniformly stable within $\mathcal{D}^{\eqref{tagA001}}_r$;  if the number of these elements
is finite, all of these elements are almost periodic.
\end{thm}

\begin{proof}
\medskip
\noindent {\bf Step\ 1. Uniform stability.}  If there is some
$\mu(t)\in\mathcal{D}^{\eqref{tagA001}}_r$ which is not uniformly stable on $[t_0,+\infty)$
within $\mathcal{D}^{\eqref{tagA001}}_r$ for some $t_0\in\mathbb{R}$, then there is a sequence
$\mu_n(t)\in\mathcal{D}^{\eqref{tagA001}}_r$
such that $\rho(\mu_n(t_0),\mu(t_0))\rightarrow 0$ and there are $t_n\geq t_0$ such that
\begin{equation}\label{claim01}
\inf_n\rho(\mu_n(t_n),\mu(t_n))\geq \epsilon_0.
\end{equation}
By Skorohod representation theorem, there exist suitable random variables $\hat{X}_n$,
$\hat{X}$ such that $\mathcal{L}(\hat{X}_n)=\mu_n(t_0)$, $\mathcal{L}(\hat{X})=\mu(t_0)$ and
$\hat{X}_n\xrightarrow{a.s.}\hat{X}$. By the global Lipschitz condition of coefficients, there
exist strong solutions $X_n(t), X(t)\in\mathcal{B}^{\eqref{tagA001}}_r$ for given Brownian motion
$W$ such that $X_n(t_0)=\hat{X}_n$, $X(t_0)=\hat{X}$, and $\mathcal{L}(X(t))=\mu(t)$,
$\mathcal{L}(X_n(t))=\mu_n(t)$.

We want to prove that $\rho(\mu_n(t_n),\mu(t_n))\rightarrow 0$ and hence
get contradiction to \eqref{claim01}.  It suffices to prove that $X_n(t)$ uniformly converge to $X(t)$ in probability
on $[t_0,+\infty)$, that is, for every $\epsilon>0$, when $n$ is large enough,
\begin{equation}\label{proba010}
\mathbf{P}\{\sup_{t\geq t_0}|X_n(t)-X(t)|\geq \epsilon\}<\epsilon.
\end{equation}

Firstly, we prove that $V(t,X_{n}(t)-X(t))$ is a supermartingale on $[t_0,+\infty)$ for each $n$.
For $t\geq t_0$, we have
\begin{equation*}
\begin{split}
X_n(t)-X(t)= & \hat{X}_n-\hat{X}+\int_{t_0}^tf(s,X_n(s))-f(s,X(s))\rmd s\\
& +\int_{t_0}^tg(s,X_n(s))-g(s,X(s))\rmd W(s).
\end{split}
\end{equation*}
For every $\epsilon>0$, let
\begin{equation}\label{defineused01}
V_{\epsilon}:=\inf_{|x|\geq \epsilon,t\geq t_0}V(t,x).
\end{equation}
By \eqref{whether01} we can see $V_{\epsilon}>0$.
For $t_0\leq s< t<+\infty$, and every $k,n\in \mathbb{N}$, we define a stopping time
$$\tau^n_k:=\inf\{t\geq s:|X_n(t)|\vee|X(t)|>k\}\text{.}$$
By It\^o's formula,
\begin{equation*}
\begin{split}
& V(\tau^{n}_k\wedge t,X_{n}(\tau^{n}_k\wedge t)-X(\tau^{n}_k\wedge t))\\
= & V(s,X_{n}(s)-X(s))+\int_{s}^{\tau^{n}_k\wedge t}\mathscr{L}V(u,X_{n}(u)-X(u))\rmd u\\
& +\int_s^{\tau^{n}_k\wedge t}\sum_{i=1}^m\sum_{j=1}^d[g_{ji}(u,X_{n}(u))-g_{ji}(u,X(u))]\frac{\partial V}{\partial x_j}(u,X_{n}(u)-X(u))\rmd W_i(u).
\end{split}
\end{equation*}
Then we have
$$\mathbf{E}(\int_s^{\tau^{n}_k\wedge t}\sum_{i=1}^m\sum_{j=1}^d[g_{ji}(u,X_{n}(u))-g_{ji}(u,X(u))]\frac{\partial V}{\partial x_j}(u,X_{n}(u)-X(u))\rmd W_i(u)|\mathcal{F}_s)=0\text{ a.s..}$$
By \eqref{oldfind03},
\begin{equation*}
\begin{split}
& \mathbf{E}(V(\tau^{n}_k\wedge t,X_{n}(\tau^{n}_k\wedge t)-X(\tau^{n}_k\wedge t))|\mathcal{F}_s)\\
= & \mathbf{E}(V(s,X_{n}(s)-X(s))|\mathcal{F}_s)+\mathbf{E}(\int_{s}^{\tau^{n}_k\wedge t}\mathscr{L}V(u,X_{n}(u)-X(u))\rmd u|\mathcal{F}_s).
\end{split}
\end{equation*}
Since $\mathcal{L}V(t,x)\leq 0$ on $\mathbb{R}\times\mathbb{R}^d$  and $V(s,X_{n}(s)-X(s))$ is $\mathcal{F}_s$-measurable, we get
\begin{equation}\label{result01}
\begin{split}
\mathbf{E}(V(\tau^{n}_k\wedge t,X_{n}(\tau^{n}_k\wedge t)-X(\tau^{n}_k\wedge t))|\mathcal{F}_s)\leq & \mathbf{E}(V(s,X_{n}(s)-X(s))|\mathcal{F}_s)\\
= & V(s,X_{n}(s)-X(s))\text{, a.s..}
\end{split}
\end{equation}
Because $V(t,x)$ is $C^2 $ in $t$, $\tau^{n}_k\xrightarrow{a.s.}+\infty$ as $k\rightarrow+\infty$
for every $n$, by Fatou's lemma we have:
\begin{equation}\label{mar}
\begin{split}
\mathbf{E}(V(t,X_n(t)-X(t))|\mathcal{F}_s)= & \mathbf{E}(\liminf_{k\rightarrow+\infty}(V(\tau^{n}_k\wedge t,X_{n}(\tau^{n}_k\wedge t)-X(\tau^{n}_k\wedge t))|\mathcal{F}_s)\\
\leq & \liminf_{k\rightarrow+\infty}\mathbf{E}V(\tau^{n}_k\wedge t,X_{n}(\tau^{n}_k\wedge t)-X(\tau^{n}_k\wedge t))|\mathcal{F}_s)\\
\leq & V(s,X_{n}(s)-X(s))\text{, a.s..}
\end{split}
\end{equation}
So $V(t,X_{n}(t)-X(t))$ is a supermartingale on $[t_0,+\infty)$.

Now we want to prove that $\mathbf{E}\sqrt{V(t_0,\hat{X}_n-\hat{X})}$ is sufficiently small when
$n$ is large enough. By Jensen's inequality and \eqref{mar} we have
\begin{equation}\label{Jen1}
\begin{split}
\mathbf{E}(\sqrt{V(t,X_n(t)-X(t))}|\mathcal{F}_s)\leq & \sqrt{\mathbf{E}(V(t,X_n(t)-X(t))|\mathcal{F}_s)}\\
\leq & \sqrt{V(s,X_n(s)-X(s))}\text{, a.s..}
\end{split}
\end{equation}
That is, $\sqrt{V(t,X_{n}(t)-X(t))}$ is a supermartingale. So by the martingale inequality we
have
\begin{equation}\label{oldfinda}
\begin{split}
\mathbf{P}\{\sup_{t\in[t_0,+\infty)}|X_n(t)-X(t)|\geq \epsilon\}\leq & \mathbf{P}\{\sup_{t\in[t_0,+\infty)}\sqrt{V(t,X_n(t)-X(t))}\geq\sqrt{V_{\epsilon}}\}\\
\leq & \frac{\mathbf{E}\sqrt{V(t_0,\hat{X}_n-\hat{X})}}{\sqrt{V_{\epsilon}}}.
\end{split}
\end{equation}
Note that $\hat{X}_n\xrightarrow{a.s.}\hat{X}$ and
$\sup_n\mathbf{E}|\hat{X}_n|^2\leq r^2$, we have (cf. \cite[Theorems $4.5.2$, $4.5.4$]{Ch}):
$$\mathbf{E}|\hat{X}_n|\rightarrow \mathbf{E}|\hat{X}|\text{, as $n\rightarrow+\infty$,}$$
and
$$\lim_{n\rightarrow+\infty}\mathbf{E}|\hat{X}_n-\hat{X}|=0.$$
By \eqref{oldfind02}, we have
$$\frac{\mathbf{E}\sqrt{V(t_0,\hat{X}_n-\hat{X})}}{\sqrt{V_{\epsilon}}}\leq\frac{\sqrt{b}\mathbf{E}|\hat{X}_n-\hat{X}|}{\sqrt{V_{\epsilon}}},$$
which implies that, if $n$ is large enough such that
$$\mathbf{E}|\hat{X}_n-\hat{X}|<\frac{\epsilon\sqrt{V_{\epsilon}}}{\sqrt{b}},$$
we will have \eqref{proba010}. Thus
$$\sup_{t\geq t_0}\rho(\mu_n(t),\mu(t))\rightarrow 0,$$
which is contradictory to \eqref{claim01}. Thus each element of $\mathcal{D}^{\eqref{tagA001}}_r$ is uniformly stable within
$\mathcal{D}^{\eqref{tagA001}}_r$.

\medskip
\noindent {\bf Step\ 2.  Inherited property and a.a.p.}\  Now we want to prove that the
consequence of step $1$ is also valid for all the hull equations.

Let the sequence $\alpha'$ be such
that $(T_{\alpha'}f,T_{\alpha'}g)$ unifromly exists on $\R\times S$ for any compact set $S\subset \R^d$.
Since $V$, $V_t$, $V_{x_i}$ are bounded on $\mathbb{R}\times S$,
$V(t+\alpha'_n,x)$ are uniformly bounded and equi-continuous on $I\times S$ for any compact interval
$I\subset\mathbb{R}$. By Arzela-Ascoli's
theorem, there is suitable subsequence $\alpha\subset\alpha'$ such that $T_{\alpha}V(t,x)$ exists
uniformly on $I\times S$. By the diagonalization argument, the $\alpha$ could be chosen such that
$T_{\alpha}V$ exists uniformly on any compact subset of $\mathbb{R}\times\mathbb{R}^d$.

Similarly we can extract further subsequence from $\alpha$, which we still denote by $\alpha$ itself,
such that $T_{\alpha}V_t$, $T_{\alpha}V_{x_i}$,
$T_{\alpha}V_{x_ix_j}$ exist uniformly on compact subsets of $\mathbb{R}\times\mathbb{R}^d$. More precisely,
we have
\begin{equation}\label{nextuse}
\frac{\partial T_{\alpha}V}{\partial t}=T_{\alpha}V_t\text{, $\frac{\partial T_{\alpha}V}{\partial x_i}=T_{\alpha}V_{x_i}$, $\frac{\partial^2 T_{\alpha}V}{\partial x_i\partial x_j}=T_{\alpha}V_{x_ix_j}$, for $i,j=1,\cdots,d$, on $\mathbb{R}\times\mathbb{R}^d$.}
\end{equation}
So we have
\begin{equation}\label{limit01}
T_{\alpha}V(t,x)\leq b|x|^2,
\end{equation}
\begin{equation}\label{limit02}
\begin{split}
\mathscr{L}T_{\alpha}V(t,x-y)= & \frac{\partial T_{\alpha}V}{\partial t}(t,x-y)+\sum_{i=1}^d\frac{\partial T_{\alpha}V}{\partial x_i}(t,x-y)(T_{\alpha}f_i(t,x)-T_{\alpha}f_i(t,y))\\
& +\frac{1}{2}\sum_{l=1}^m\sum_{i,j=1}^d(T_{\alpha}g_{il}(t,x)-T_{\alpha}g_{il}(t,y))\frac{\partial^2 T_{\alpha}V}{\partial x_i\partial x_j}(t,x-y)\\
& \cdot(T_{\alpha}g_{jl}(t,x)-T_{\alpha}g_{jl}(t,y)))\leq 0
\end{split}
\end{equation}
for all $(t,x,y)\in\mathbb{R}\times\mathbb{R}^d\times\mathbb{R}^d$.
Repeating Step $1$,  we obtain that all the elements of $\mathcal{D}^{(T_{\alpha}f,T_{\alpha}g)}_r$ are
uniformly stable within $\mathcal{D}^{(T_{\alpha}f,T_{\alpha}g)}_r$.

By the uniform stability and the finiteness of the set $\mathcal{D}_r^{\eqref{tagA001}}$, there is a separating constant $d(f,g)$, depending only on $(f,g)$
but independent of $\mu\in\mathcal{D}_r^{\eqref{tagA001}}$, such that for any two different elements
$\eta(t),\mu(t)\in\mathcal{D}^{\eqref{tagA001}}_r$ we have
\begin{equation}\label{sep}
\inf_{t\in\mathbb{R}_-}\rho(\eta(t),\mu(t))> d(f,g).
\end{equation}
By Proposition \ref{uap-p2}-(ii), we may assume with loss of generality that the above sequence $\alpha$
satisfies $\lim_{n\to\infty} \alpha_n =-\infty$, so it follows from \eqref{sep} that
\[
\inf_{t\in\mathbb{R}_-}\rho(T_\alpha\eta(t),T_\alpha \mu(t))\ge d(f,g).
\]
On the other hand, it follows from Proposition \ref{key} that
$T_{\alpha}\mu(t)\in\mathcal{D}^{(T_{\alpha}f,T_{\alpha}g)}_r$,
so $\mathcal{D}^{(T_{\alpha}f,T_{\alpha}g)}_r$ has no less elements than
$\mathcal{D}^{\eqref{tagA001}}_r$ does.

Conversely, by Proposition \ref{uap-p2}-(i), $(T_{\alpha}f,T_{\alpha}g)$ is uniformly almost periodic and $(f,g)\in H(T_{\alpha}f,T_{\alpha}g)$.
So, by the same symmetric argument as above, $\mathcal{D}^{\eqref{tagA001}}_r$ also has no less elements than
$\mathcal{D}^{(T_{\alpha}f,T_{\alpha}g)}_r$ does and the separating constant $d(T_\alpha f, T_\alpha g) \le d (f,g)$. That is,
all the equations in the hull $H(f,g)$ share the same number of elements as $\mathcal{D}^{\eqref{tagA001}}_r$ and the same separating constant $d(f,g)$.

Now we prove that all the elements of $\mathcal{D}^{\eqref{tagA001}}_r$ are a.a.p..
For the above sequence $\alpha$ with  $\alpha_n\rightarrow-\infty$ and given sequence $\delta=\{\delta_n\}$ with $\delta_n<0$, by Proposition \ref{uap-p2}-(iii)
there exist suitable subsequences which we denote as themselves such that $(T_{\alpha+\delta}f,T_{\alpha+\delta}g)$
exists uniformly on $\R\times S$ for any compact set $S\subset\R^d$.
By Arzela-Ascoli's theorem there are subsequences $\beta,\gamma\subset\alpha$ such that
$T_{\beta+\delta}\mu(t)$, $T_{\gamma+\delta}\mu(t)$ exist uniformly on compact intervals (see the proof of  \cite[Theorem 3.1]{LW} for details).
By Proposition \ref{key},
$T_{\beta+\delta}\mu(t),T_{\gamma+\delta}\mu(t)\in\mathcal{D}^{(T_{\alpha+\delta}f,T_{\alpha+\delta}g)}_r$, then
by the separating property obtained above we have
$$
T_{\beta+\delta}\mu(t)\equiv T_{\gamma+\delta}\mu(t)\text{ or }\inf_{t\in\mathbb{R}_-}\rho(T_{\beta+\delta}\mu(t),T_{\gamma+\delta}\mu(t))\geq d(f,g).
$$
Then it follows from Lemma \ref{uselessmust01} that all the elements of $\mathcal{D}^{\eqref{tagA001}}_r$ are all a.a.p. on $\mathbb{R}_-$.

By Proposition \ref{im}, there
is some $\hat{\mu}(t)\in\mathcal{D}_r^{\eqref{tagA001}}$, which is almost periodic and satisfies
$$
\lim_{t\rightarrow-\infty}\rho(\mu(t),\hat{\mu}(t))=0.
$$
By the separating property, $\hat{\mu}(t)=\mu(t)$, which implies that each element of
$\mathcal{D}_r^{\eqref{tagA001}}$ is almost periodic. The proof is complete.
\end{proof}

The following result, which limits the number of $\mathcal{D}^{(f,g)}_r$'s elements to be one, is an important special case of
Theorem \ref{oldfind01} and is more convenient for use in applications.

\begin{thm}\label{stab01}
Suppose that \eqref{tagA001}'s coefficients satisfy condition {\bf (H)}. Assume that there is a
function $V(\cdot,\cdot)$ satisfying condition {\bf (L)}, and there are constants
$a,b>0$ such that
\begin{equation}\label{tagh_0}
a|x|^2\leq V(t,x)\leq b|x|^2 \quad \hbox{for all } (t,x)\in\mathbb{R}\times\mathbb{R}^d.
\end{equation}
Assume that there is some positively definite function $c(\cdot):\mathbb{R}_+\rightarrow\mathbb{R}_+$
which is convex, increasing on $\mathbb{R}_+$, and
\begin{equation}\label{tagh_1}
\mathscr{L}V(t,x-y)\leq-c(|x-y|^2)\quad \text{ $\forall t\in\mathbb{R}$, $\forall x,y\in\mathbb{R}^d$.}
\end{equation}
Then if $\mathcal{D}^{\eqref{tagA001}}\neq\emptyset$, it has a unique element which is almost
periodic.
\end{thm}

\begin{proof}
We prove the uniqueness by contradiction.
If there are two elements $\mu(t),\eta(t)\in\mathcal{D}^{\eqref{tagA001}}$, then there's some $r>0$ such that
$\mu(t),\eta(t)\in\mathcal{D}_r^{\eqref{tagA001}}$. Assume that $X(t)$, $Y(t)$ are two strong
$\mathbf{L}^2$-bounded solutions of \eqref{tagA001} for given Brownian motion $W(t)$ such that
$\mathcal{L}(X(t))=\mu(t)$, $\mathcal{L}(Y(t))=\eta(t)$.

For given $\epsilon>0$, let $T(\epsilon)=2br^2/c(\epsilon)+1$. Firstly, we prove that for every $t\in\mathbb{R}$ there is $t_1\in[t,t+T(\epsilon)]$
such that
\begin{equation}\label{time01}
\mathbf{E}|X(t_1)-Y(t_1)|^2<\epsilon.
\end{equation}
If this is not true, then there is $\hat{t}\in\mathbb{R}$ and $\epsilon_0>0$ such
that
$$\inf_{t\in[\hat{t},\hat{t}+T(\epsilon_0)]}\mathbf{E}|X(t)-Y(t)|^2\geq\epsilon_0.$$
Similar to the proof of Theorem \ref{oldfind01}, for given $s\in\mathbb{R}$, we define
$$
\tau_k:=\inf\{t\geq s:|Y(t)|\vee|X(t)|>k\}.
$$
Then it follows from Ito's formula that for $t\geq s$,
\begin{equation*}
\begin{split}
& V(\tau_k\wedge t,X(\tau_k\wedge t)-Y(\tau_k\wedge t))\\
= & \mathbf{E}V(s,X(s)-Y(s))+\int_{s}^{\tau_k\wedge t}\mathbf{E}\mathscr{L}V(u,X(u)-Y(u))\rmd s\\
& +\int_s^{\tau^{n}_k\wedge t}\sum_{i=1}^m\sum_{j=1}^d[g_{ji}(u,X_{n}(u))-g_{ji}(u,X(u))]\frac{\partial V}{\partial x_j}(u,X_{n}(u)-X(u))\rmd W_i(u).
\end{split}
\end{equation*}
Since
$$
\sup_{t\in \mathbb R}\mathbf{E}|X(t)-Y(t)|^2\leq 2r^2,
$$
by \eqref{tagh_0}, \eqref{tagh_1} we have
\begin{equation*}
\begin{split}
& \mathbf{E}V(\tau_k\wedge t,X(\tau_k\wedge t)-Y(\tau_k\wedge t))\\
\leq & b\mathbf{E}|X(s)-Y(s)|^2-\mathbf{E}\int_{s}^{\tau_k\wedge t}c(|X(u)-Y(u)|^2)\rmd s\\
\leq & 2br^2-\mathbf{E}\int_{s}^{\tau_k\wedge t}c(|X(u)-Y(u)|^2)\rmd s.
\end{split}
\end{equation*}
Because $c(r)$ is convex, increasing on $\mathbb{R}_+$, by Jensen's inequality we have:
$$\mathbf{E}\int_{s}^{\tau_k\wedge t}c(|X(u)-Y(u)|^2)\rmd s\geq\int_{s}^{\tau_k\wedge t}c(\mathbf{E}|X(u)-Y(u)|^2)\rmd s\geq c(\epsilon_0)(\tau_k\wedge t-s).$$
So
\begin{equation}\label{time02}
\mathbf{E}V(\tau_k\wedge t,X(\tau_k\wedge t)-Y(\tau_k\wedge t))\leq 2br^2-c(\epsilon_0)(\tau_k\wedge t-s).
\end{equation}
Noting that $\tau_k\xrightarrow{a.s.} +\infty$ as $k\rightarrow+\infty$,
by Fatou's lemma and \eqref{time02} we have
\begin{equation*}
\begin{split}
\mathbf{E}V(t,X(t)-Y(t))& =  \mathbf{E}(\liminf_{k\rightarrow+\infty}V(\tau_k\wedge t,X(\tau_k\wedge t)-Y(\tau_k\wedge t)))\\
& \leq  \liminf_{k\rightarrow+\infty}\mathbf{E}[2br^2-c(\epsilon_0)(\tau_k\wedge t-s)]\\
& \leq  2br^2-c(\epsilon_0)(t-s).
\end{split}
\end{equation*}
Letting $s=\hat{t}$ and $t=\hat{t}+T(\epsilon_0)$, we have
$$0\leq\mathbf{E}V(\hat{t}+T(\epsilon_0),X(\hat{t}+T(\epsilon_0))-Y(\hat{t}+T(\epsilon_0)))\leq 2br^2-c(\epsilon_0)T(\epsilon_0)=-c(\epsilon_0)<0,$$
a contradiction. Thus there is $t_1\in[t,t+T(\epsilon)]$ such that \eqref{time01}
is valid.

For given $s\in\mathbb{R}$, assume $t_1\in[s,s+T(\epsilon)]$ fulfills
\eqref{time01}. By \eqref{tagh_0}--\eqref{time01}, for any $t\geq t_1$,  we have:
$$a\mathbf{E}|X(t)-Y(t)|^2\leq \mathbf{E}V(t,X(t)-Y(t))\leq \mathbf{E}V(t_1,X(t_1)-Y(t_1))\leq b\epsilon.$$
Note that $s\in\mathbb{R}$ is arbitrarily chosen and $T(\epsilon)$ is only determined by
$\epsilon>0$, so we actually have proved
$$\mathbf{E}|X(t)-Y(t)|^2\leq\epsilon\qquad \hbox{for all } t\in\R. $$
Thus $X(t)=Y(t)$ for all $t\in\R$ almost surely, which implies that $\mu(t)=\eta(t)$ for all $t\in\R$. That is,
$\mathcal{D}^{\eqref{tagA001}}$ has a unique element if it is not
empty. Finally, it follows from Theorem \ref{oldfind01} that this unique element is almost periodic. The proof is complete.
\end{proof}

Now we give another result for the existence of $\mathcal{D}^{\eqref{tagA001}}$'s
almost periodic elements without the information of the number of its elements.

\begin{thm}\label{mayber001}
Assume that \eqref{tagA001}'s coefficients satisfy condition {\bf (H)}, and there exists a
function
$V(\cdot,\cdot)$ satisfying condition {\bf (L)}. Suppose that there is some constant
$b>0$ such that \eqref{oldfind02} is valid on $\mathbb{R}_+\times\mathbb{R}^d$ and
for all $t\in\mathbb{R}_+$, $s_1, s_2\in\mathbb{R}_+$ and $x,y\in\mathbb{R}^d$,
\begin{align}\label{mayber002}
\mathscr{L}_{s_1,s_2}V(t,x-y):= & \frac{\partial V}{\partial t}(t,x-y)+\sum_{i=1}^d\frac{\partial V}{\partial x_i}(t,x-y)(f_i(t+s_1,x)-f_i(t+s_2,y))\notag\\
& +\frac{1}{2}\sum_{l=1}^m\sum_{i,j=1}^d(g_{il}(t+s_1,x)-g_{il}(t+s_2,y))\frac{\partial^2 V}{\partial x_i\partial x_j}(t,x-y)\\
& \cdot(g_{jl}(t+s_1,x)-g_{jl}(t+s_2,y))\leq 0\text{.}\notag
\end{align}
Then if \eqref{tagA001} has $\mathbf{L}^2$-bounded solutions, the distributions of these solutions
are a.a.p. on $\mathbb{R}_+$ and \eqref{tagA001} admits at least one solution with almost
periodic distribution.
\end{thm}

\begin{proof}
For sequence $\alpha=\{\alpha_n\}$ such that $\alpha_n\rightarrow+\infty$, assume that
$(T_{\alpha}f,T_{\alpha}g)$ exist uniformly on $\mathbb{R}\times S$ for any compact set
$S\subset\mathbb{R}^d$,
and $T_{\alpha}\mu(t)$ exists uniformly on compact intervals  (see, again, the proof of  \cite[Theorem 3.1]{LW} for details).
For $r>0$ and every $\mu(t)\in\mathcal{D}_r^{\eqref{tagA001}}$, we what to prove that
$T_{\alpha}\mu(t)$ uniformly exists on $\mathbb{R}_+$.

By Skorohod representation theorem,
there are suitable random variables $\hat{X}_n$, $\hat{X}$ such that
$\hat{X}_n\xrightarrow{a.s.}\hat{X}$ as $n\rightarrow+\infty$ and
$\mathcal{L}(\hat{X}_n)=\mu(\alpha_n)$, $\mathcal{L}(\hat{X})=T_{\alpha}\mu(0)$.
By the global Lipschitz condition of coefficients, for given Brownian motion $W(t)$,
there are strong solutions
$X_n(t)\in\mathcal{B}_r^{(f_{\alpha_n},g_{\alpha_n})}$ such that $X_n(0)=\hat{X}_n$, and
$\mathcal{L}(X_n(t))=\mu(t+\alpha_n)$. And for every $n,p\in\mathbb{N}$, we have
\begin{equation*}
\begin{split}
X_{n+p}(t)-X_n(t)= &~ \hat{X}_{n+p}-\hat{X}_n+\int_0^t(f(u+\alpha_{n+p}),X_{n+p}(u))-f(u+\alpha_n,X_n(u)))\rmd u\\
& ~ +\int_0^t(g(u+\alpha_{n+p},X_{n+p}(u))-g(u+\alpha_n,X_n(u)))\rmd W(u).
\end{split}
\end{equation*}

We now show that $V(t,X_{n+p}(t)-X_n(t))$ is a supermartingale
on $\mathbb{R}_+$ for given $n$ and $p$. For every $0\leq s< t<+\infty$, we define stopping times
$$\tau^{n,p}_k:=\inf\{t\geq s;|X_n(t)|\vee|X_{n+p}(t)|>k\}\text{, for every $k,n,p\in\mathbb{N}$.}$$
By It\^o's formula we have
\begin{equation*}
\begin{split}
& V(\tau^{n,p}_k\wedge t,X_{n+p}(\tau^{n,p}_k\wedge t)-X_n(\tau^{n,p}_k\wedge t))\\
= & V(s,X_{n+p}(s)-X_n(s))+\int_{s}^{\tau^{n,p}_k\wedge t}\mathscr{L}_{\alpha_{n+p},\alpha_n}V(u,X_{n+p}(u)-X_n(u))\rmd u\\
& +\int_s^{\tau^{n,p}_k\wedge t}\sum_{i=1}^m\sum_{j=1}^d[g_{ji}(u+\alpha_{n+p},X_{n+p}(u))-g_{ji}(u+\alpha_{n},X_n(u))]\\
& \cdot\frac{\partial V}{\partial x_j}(u,X_{n+p}(u)-X_n(u))\rmd W_i(u).
\end{split}
\end{equation*}
Since
\begin{equation*}
\begin{split}
& \mathbf{E}\bigg(\int_s^{\tau^{n,p}_k\wedge t}\sum_{i=1}^m\sum_{j=1}^d[g_{ji}(u+\alpha_{n+p},X_{n+p}(u))-g_{ji}(u+\alpha_{n},X_n(u))]\\
&\qquad  \cdot\frac{\partial V}{\partial x_j}(u,X_{n+p}(u)-X_n(u))\rmd W_i(u)|\mathcal{F}_s\bigg)\\
=& 0\text{, a.s.}
\end{split}
\end{equation*}
it follows from \eqref{mayber002} that
\begin{equation*}
\begin{split}
& \mathbf{E}(V(\tau^{n,p}_k\wedge t,X_{n+p}(\tau^{n,p}_k\wedge t)-X_n(\tau^{n,p}_k\wedge t))|\mathcal{F}_s)\\
\leq & \mathbf{E}(V(s,X_{n+p}(s)-X_n(s))|\mathcal{F}_s)\\
= & V(s,X_{n+p}(s)-X_n(s))\text{, a.s..}
\end{split}
\end{equation*}
Noting that $\tau^{n,p}_k\xrightarrow{a.s.}+\infty$ as
$k\rightarrow+\infty$ for every $n$, $p$, we have by Fatou's lemma
(similar to \eqref{mar}):
\begin{equation*}
\begin{split}
\mathbf{E}(V(t,X_{n+p}(t)-X_n(t))|\mathcal{F}_s)= & \mathbf{E}(\liminf_{k\rightarrow+\infty}V(\tau^{n,p}_k\wedge t,X_{n+p}(\tau^{n,p}_k\wedge t)-X_n(\tau^{n,p}_k\wedge t)|\mathcal{F}_s)\\
\leq & \liminf_{k\rightarrow+\infty}\mathbf{E}V(\tau^{n,p}_k\wedge t,X_{n+p}(\tau^{n,p}_k\wedge t)-X_n(\tau^{n,p}_k\wedge t)|\mathcal{F}_s)\\
\leq & V(s,X_{n+p}(s)-X_n(s))\text{, a.s..}
\end{split}
\end{equation*}
That is, $V(t,X_{n+p}(t)-X_n(t))$ is a supermartingale for given $p$ and $n$.
Similar to \eqref{Jen1}, by Jensen's inequality,
\begin{equation*}
\begin{split}
\mathbf{E}\left(\sqrt{V(t,X_{n+p}(t)-X_n(t))}~|\mathcal{F}_s\right) \leq & \sqrt{\mathbf{E}(V(t,X_{n+p}(t)-X_n(t))|\mathcal{F}_s)}\\
\leq & \sqrt{V(s,X_{n+p}(s)-X_n(s))}\text{, a.s..}
\end{split}
\end{equation*}
So $\sqrt{V(t,X_{n+p}(t)-X_n(t))}$ is also a supermartingale.

For any $\epsilon>0$,
we define $V_{\epsilon}>0$ as the one in \eqref{defineused01}. Then by the martingale inequality,
we have
\begin{equation}\label{want02}
\begin{split}
\mathbf{P}\left\{\sup_{t\in\mathbb{R}_+}|X_{n+p}(t)-X_n(t)|\geq\epsilon\right\}\leq & \mathbf{P}\left\{\sup_{t\in\mathbb{R}_+}\sqrt{V(t,X_{n+p}(t)-X_n(t))}\geq \sqrt{V_{\epsilon}}\right\}\\
\leq & \frac{\mathbf{E}\sqrt{V(0,\hat{X}_{n+p}-\hat{X}_n)}}{\sqrt{V_{\epsilon}}}.
\end{split}
\end{equation}
Because $\mathbf{E}|\hat{X}_n|^2\leq r^2$ and $\hat{X}_n\xrightarrow{a.s.}\hat{X}$, we have
(see \cite[Theorems $4.5.2$, $4.5.4$]{Ch}),
$$\mathbf{E}|\hat{X}_n|\rightarrow \mathbf{E}|\hat{X}|\text{, as $n\rightarrow+\infty$,}$$
and
$$\lim_{n\rightarrow+\infty}\mathbf{E}|\hat{X}_n-\hat{X}|=0.$$
So
$$\lim_{n\rightarrow+\infty}\sup_{p\in\mathbb{N}}\mathbf{E}|\hat{X}_{n+p}-\hat{X}_n|=0.$$
When $n$ is large enough, by \eqref{oldfind02} we have
$$\sup_{p\in\mathbb{N}}\mathbf{E}\sqrt{V(0,\hat{X}_{n+p}-\hat{X}_n)}\leq\sqrt{b}\sup_{p\in\mathbb{N}}\mathbf{E}|\hat{X}_{n+p}-\hat{X}_n|<\epsilon\sqrt{V_{\epsilon}}.$$
This together with \eqref{want02} implies
\begin{equation*}
\sup_{p\in\mathbb{N}}\mathbf{P}\{\sup_{t\in\mathbb{R}_+}|X_{n+p}(t)-X_n(t)|\geq\epsilon\}<\epsilon.
\end{equation*}
By Theorem $4.1.3$ in \cite{Ch}, there exists a suitable stochastic process $\widetilde{X}(t)$
such that $X_n(t)\xrightarrow{\mathbf{P}}\widetilde{X}(t)$ uniformly on $\mathbb{R}_+$.
Thus $\mu(t+\alpha_n)$ uniformly
converges to some
$T_{\alpha}\mu(t)$ on $\mathbb{R}_+$.

By Lemma \ref{uselessmust01}, we can see that each $\mu(t)\in\cal D_r^{\eqref{tagA001}}$ is a.a.p. on
$\mathbb{R}_+$.
So the distribution of any $\mathbf{L}^2$-bounded solution
is a.a.p. on $\mathbb{R}_+$. By Proposition \ref{im}, there exists some $\mathbf{L}^2$-bounded
solution of \eqref{tagA001} with almost periodic distribution. The proof is complete.
\end{proof}




To discuss the almost periodicity of SDE's solutions,  we need to find ways to obtain
$\mathbf{L}^2$-bounded solutions on $\R$, which may reduce to finding $\mathbf{L}^2$-bounded solutions on $\R_+$:

\begin{prop}[cf. \cite{LW}, Theorem $4.7$]\label{appl01}
Assume that \eqref{tagA001}'s coefficients satisfy condition {\bf (H)}, \eqref{tagA001} admits a
solution $\varphi$ on $[t_0,+\infty)$ for some $t_0\in\mathbb{R}$, and
$\sup_{t\geq t_0}||\varphi(t)||_2\leq M$ for some constant $M>0$, then \eqref{tagA001} has a solution
$\widetilde{\varphi}$ on $\mathbb{R}$ with
$||\widetilde{\varphi}(t)||_{\infty}\leq M$.
\end{prop}

We now conclude this section by giving a sufficient condition for the existence of $\mathbf{L}^2$-bounded solutions via Lyapunov functions:

\begin{thm}\label{appl02}
Assume that \eqref{tagA001}'s coefficients satisfy condition {\bf (H)}, and there is a function
$V$ satisfying condition {\bf (L)} such that for some constant $R>0$
$$a|x|^2\leq V(t,x)\leq b(t)|x|^2+c(t)\text{, when $|x|\leq R$,}$$
where constant $a>0$, $b(\cdot)$, $c(\cdot)$ are positive functions on $\mathbb{R}$. Assume further that
$$LV(t,x):=\frac{\partial V}{\partial t}+\sum_{i=1}^d\frac{\partial V}{\partial x_i}f_i+\sum_{l=1}^m\sum_{i,j=1}^dg_{il}\frac{\partial^2 V}{\partial x_i\partial x_j}g_{jl}\leq 0\text{, when $|x|\geq R$.}$$
Then if $X(t)$ is a solution of \eqref{tagA001} with initial condition $\mathbf{E}|X(t_0)|^2<+\infty$, $X(t)$ is $\mathbf{L}^2$-bounded on $[t_0,+\infty)$.
\end{thm}

\begin{proof}
Suppose that $X(t)$ is the solution of \eqref{tagA001} with $\mathbf{L}^2$-bounded initial value at $t_0$.
Since the coefficients satisfy condition {\bf (H)}, $X(t)$ exists on $[t_0,+\infty)$. We define
a sequence of stopping times:
$$\tau_n^R:=\inf\{t\geq t_0: |X(t)|\geq n\text{, or $|X(t)|\leq R$}\},$$
and
$$\tau^R:=\inf\{t\geq t_0: |X(t)|\leq R\}.$$
Then $\tau_n^R\xrightarrow{a.s.}\tau^R$ as $n\rightarrow+\infty$.

Denote $B_R$ as the close ball $\{x\in\mathbb{R}^d: |x|\leq R\}$.
When $X(t_0)$ is supported on $\mathbb{R}^d-B_R$, by It\^o's formula, for
$t\geq t_0$,
\begin{equation*}
\begin{split}
\mathbf{E}V(t\wedge\tau^R_n,X(t\wedge\tau^R_n))= & \mathbf{E}V(t_0,X(t_0))+\mathbf{E}\int_{t_0}^{t\wedge\tau^R_n}LV(u,X(u))\rmd u\\
\leq & \mathbf{E}V(t_0,X(t_0))\leq c(t_0)+b(t_0)\mathbf{E}|X(t_0)|^2.
\end{split}
\end{equation*}
Then Fatou's lemma implies that
\begin{equation}\label{case1}
\mathbf{E}V(t\wedge\tau^R,X(t\wedge\tau^R))\leq \mathbf{E}V(t_0,X(t_0))\leq c(t_0)+b(t_0)\mathbf{E}|X(t_0)|^2,
\end{equation}
by letting $n\rightarrow+\infty$.

When $X(t_0)$ is supported on $\mathbb{R}^d$, denote $\bar{M}$ as a bound of $V(t,x)$ for
$|x|\leq R$, then by \eqref{case1} we have for $t\ge t_0$
\begin{equation}\label{boundedsolution}
\begin{split}
\mathbf{E}V(t,X(t))\leq & ~\mathbf{P}(\tau^R\geq t)\cdot\int_{\{X(t_0)>R\}}V(t_0,X(t_0,\omega))\rmd\mathbf{P}(\omega)\\
& ~+\mathbf{P}(\tau^R< t)\cdot \left[\bar{M}+\int_{\{X(t_0)>R\}}V(t_0,X(t_0,\omega))\rmd\mathbf{P}(\omega)\right]\\
\leq &~ 2[c(t_0)+b(t_0)\mathbf{E}|X(t_0)|^2]+\bar{M}.
\end{split}
\end{equation}
Note that in \eqref{boundedsolution} either $|X(t)|\leq R$ or $a|X(t)|^2\leq V(t,X(t))$, so $X(t)$ is
$\mathbf{L}^2$-bounded on $[t_0,+\infty)$.
\end{proof}

\section{Applications}

In this section, we illustrate our theoretical results by several examples.
Firstly we consider the simplest case of almost periodic SDEs.

\begin{exa}\label{exa01a}
Consider one-dimentional SDE
\begin{equation}\label{exaeq1}
\rmd X(t)=f(t,X(t))\rmd t+g(t,X(t))\rmd W(t)\text{,}
\end{equation}
where $f$, $g$ satisfy condition {\bf (H)} and  are $C^1$ in $x$. Assume that for some constant $c>0$,
\begin{equation}\label{exaeq2}
\sup_{(t,x)\in\mathbb{R}\times\mathbb{R}}\left|\frac{\partial g}{\partial x}(t,x)\right|^2\leq c\text{, \quad }\sup_{(t,x)\in\mathbb{R}\times\mathbb{R}}\frac{\partial f}{\partial x}(t,x)\leq -c\text{.}
\end{equation}
Then if $\mathcal{D}^{\eqref{exaeq1}}\neq \emptyset$, it has a  unique element which is almost
periodic.
\end{exa}
\begin{proof}
Let $V(t,x)=|x|^2$. Then it's easy to see that $V$ satisfies condition {\bf (L)}, and
$$\frac{\partial V}{\partial t}(t,x)=0,\quad  \frac{\partial V}{\partial x}(t,x)=2x, \quad \frac{\partial^2 V}{\partial x^2}(t,x)=2.$$
By \eqref{exaeq2} and mean value theorem, for every $x,y\in\mathbb{R}$
and every $t\in\mathbb{R}$, if $x\neq y$, there exist $\hat{\xi}=\hat{\xi}(t,x,y)$,
 $\xi=\xi(t,x,y)$ such that $\hat{\xi},\xi\in(x\wedge y,x\vee y)$, and
$$(f(t,x)-f(t,y))(x-y)=\frac{\partial f}{\partial x}(t,\xi)(x-y)^2\leq -c(x-y)^2,$$
$$(g(t,x)-g(t,y))^2=(x-y)^2|\frac{\partial g}{\partial x}(t,\hat{\xi})|^2\leq c(x-y)^2.$$
So
\begin{equation*}
\begin{split}
\mathscr{L}V(t,x-y)= & 2(f(t,x)-f(t,y))(x-y)+(g(t,x)-g(t,y))^2\\
\leq & -c(x-y)^2=-c|x-y|^2.
\end{split}
\end{equation*}

By Theorem \ref{stab01} we can easily get the required result.
\end{proof}

Now let us consider some two-dimensional applications.

\begin{exa}\label{second007}
Consider two-dimentional SDE:
\begin{equation}\label{second001}
\begin{cases}
\rmd X_1(t)=[f_1(t,X_1(t))+\sigma X_2(t)]\rmd t+[A_1(t)X_1(t)+g_1(t)]\rmd W_1(t),\\
\rmd X_2(t)=[f_2(t,X_2(t))-\sigma X_1(t)]\rmd t+[A_2(t)X_2(t)+g_2(t)]\rmd W_2(t),
\end{cases}
\end{equation}
where $f_i(t,x)$ are $C^1$ in $x$ and satisfy condition {\bf (H)} for $i=1,2$. $\sigma\neq 0$ is a constant.
Assume that $A_i$, $g_i$ are almost periodic and $f_i(t,0)\equiv 0$, $i=1,2$. Denote $a(t):=\max_{i=1,2}\{A_i^2(t),g_i^2(t)\}$.  Assume further that for $t,x\in\mathbb{R}$,
\begin{equation}\label{exa02a}
\frac{\partial f_i}{\partial x}(t,x)\leq -2a(t)-1, \quad  i=1,2.
\end{equation}
Then $\mathcal{D}^{\eqref{second001}}_r$ has a unique element which is almost periodic.
\end{exa}

\begin{proof}
Let $V(\cdot,\cdot):\mathbb{R}\times\mathbb{R}^2\rightarrow\mathbb{R}$, $V(t,x)=|x|^2=x_1^2+x_2^2$.
Then $V(t,x)$ satisfies condition {\bf (L)}, and  for $(t,x)\in\mathbb{R}\times\mathbb{R}^2$,
$i,j=1,2$,
$$\frac{\partial V}{\partial x_i}(t,x)=2x_i,\quad  \frac{\partial^2V}{\partial x_i^2}(t,x)=2, \quad  \frac{\partial^2V}{\partial x_i\partial x_j}(t,x)=0\text{, when $i\neq j$.}$$
By \eqref{exa02a} and mean value theorem, for $ x=(x_1,x_2), y=(y_1,y_2)\in\mathbb{R}^2$, there are
$\xi_i=\xi_i(t,x_i,y_i)\in(x_i\wedge y_i,x_i\vee y_i)$, $i=1,2$, such that
\begin{equation*}
\begin{split}
\mathscr{L}V(t,x-y)= & 2\sum_{i=1,2}(f_i(t,x_i)-f_i(t,y_i)(x_i-y_i)+\sum_{i=1,2}A_i^2(t)(x_i-y_i)^2\\
\leq & \sum_{i=1,2}[a(t)+2\frac{\partial f}{\partial x}(t,\xi_i)(x_i-y_i)^2]\\
\leq & (-3a(t)-2)(x_i-y_i)^2\leq -2|x-y|^2\text{.}
\end{split}
\end{equation*}

Since $f_i(t,0)=0$, for every $x_i$, $t$, there exist $\hat{\xi}_{i}=\hat{\xi}_{i}(t,x_i)\in(x_i\wedge 0,x_i\vee 0)$ such that
$$f_i(t,x_i)x_i=\frac{\partial f_i}{\partial x_i}(t,\hat{\xi}_{i})x_i^2\leq -(2a(t)+1)x_i^2.$$
So
\begin{equation*}
\begin{split}
LV(t,x)= & 2\sum_{i=1,2}f_i(t,x_i)x_i+\sum_{i=1,2}[A_i(t)x_i+g_i(t)]^2\\
\leq & \sum_{i=1,2}[2A_i^2(t)x^2_i+2g_i^2(t)+2\frac{\partial f}{\partial x}(t,\hat{\xi}_i)x_i^2]\\
\leq & \sum_{i=1,2}[-(2a(t)+2)x^2_i+2a(t)].
\end{split}
\end{equation*}
Obviously $LV(t,x)\leq 0$ when $|x|\geq \sqrt{2}$. By the global Lipschitz condition of the coefficients,
we can see that \eqref{second001} must have $\mathbf{L}^2$-bounded solutions from Proposition
\ref{appl01}
and Theorem \ref{appl02}. By Theorem \ref{stab01}, we can get the result required.
\end{proof}

\begin{exa}\label{examplefinal}
Consider two-dimentional SDE:
\begin{equation}\label{example001}
\begin{cases}
\rmd X_1(t)= & [-(A_1^2(t)+A_2^2(t)+1)X_1(t)+2A_1^2(t)X_2(t)]\rmd t\\
& +A_1(t)(X_1(t)-X_2(t))\rmd W_1(t),\\
\rmd X_2(t)= & [-(A_2^2(t)+A_1^2(t)+1)X_2(t)+2A_2^2(t)X_1(t)]\rmd t\\
& +A_2(t)(X_1(t)-X_2(t))\rmd W_2(t).\\
\end{cases}
\end{equation}
If $A_i(t)$ are almost periodic for $i=1,2$,
then \eqref{example001} has $\mathbf{L}^2$-bounded solutions, and all the $\mathbf{L}^2$-bounded solutions of \eqref{example001} have the same distribution which is
almost periodic.
\end{exa}

\begin{proof}
Similar to the proof of Example
\ref{second007}, let $V(t,x)=x_1^2+x_2^2$.
For $t\in\mathbb{R}$, $x=(x_1,x_2),y=(y_1,y_2)\in\mathbb{R}^2$, we have
\begin{equation*}
\begin{split}
\mathscr{L}V(t,x-y)= & 2\sum_{i=1,2}\left[-(A_1^2(t)+A_2^2(t)+1)(x_i-y_i)^2+A_i^2(t)(x_1-y_1)(x_2-y_2)\right]\\
& +\sum_{i=1,2}[(A_1^2(t)+A_2^2(t))(x_i-y_i)^2]-2(A_1^2(t)+A_2^2(t))(x_1-y_1)(x_2-y_2)\\
\leq & -2|x-y|^2,
\end{split}
\end{equation*}
and
\begin{equation*}
\begin{split}
LV(t,x)& =  2\sum_{i=1,2}[-(A_1^2(t)+A_2^2(t)+1)x_i^2+2A_i^2(t)x_1x_2]+\sum_{i=1,2}A_i^2(t)(x_1-x_2)^2\\
& \leq -2|x|^2\leq 0.
\end{split}
\end{equation*}
By Proposition \ref{appl01} and Theorem \ref{appl02}, \eqref{example001} has $\mathbf{L}^2$-bounded solutions.
By Theorem \ref{stab01}, $\mathcal{D}^{\eqref{example001}}$ has a unique element which is almost periodic.
\end{proof}

\end{document}